\documentclass[12pt]{article}
\usepackage[centertags]{amsmath}
\usepackage{amsfonts}
\usepackage{amssymb}
\usepackage{amsthm}
\input{amssym.def}
\usepackage{lineno}

\newtheorem{Definition}{Definition}[section]
\newtheorem{Theorem}[Definition]{Theorem}
\newtheorem{Lemma}[Definition]{Lemma}

\newtheorem{Corollary}[Definition]{Corollary}

\newcommand{\lc}{\mathcal{L}}
\newcommand{\rc}{\mathcal{R}}
\newcommand{\hc}{\mathcal{H}}
\newcommand{\jc}{\mathcal{J}}

\title{\Large \bf On right $\pi$-inverse ordered semigroups}

\author{A. Jamadar \\
\footnotesize{Department of Mathematics, Rampurhat College}\\
\footnotesize{Rampurhat-731224, West Bengal, India}\\
\footnotesize{amlanjamadar@gmail.com}}

\begin{document}
\date{}
\maketitle

\begin{abstract}{\footnotesize}
Here we introduce the notion of (left, right) $\pi$-$t$-simple,
right $\pi$-inverse ordered semigroups and discuss
characterizations and relationships concerning them. Semilattice
decomposition of left $\pi$-$t$-simple ordered semigroups has been
given here. Furthermore we study an interrelation of between the
generalized Green's relations and the class of semigroups which
are semilattices of right $\pi$-$t$-simple ordered semigroups.

 \end{abstract}

\section{Introduction and Preliminaries}
A semigroup  $(S,\cdot)$ with an order relation $\leq$ is called
an ordered semigroup ($\cite{Ke2006}$) if for all $a,\;b , \;x \in
S, \;a \leq b$ implies $xa\leq xb \;\textrm{and} \;a x \leq b x $.
It is denoted by $(S,\cdot, \leq)$. Let $(S,\cdot, \leq)$ be an
ordered semigroup. For a subset $A$ of $S$, let $(A]= \{x\in S:
x\leq a, \;\textrm{for some} \;a\in A\}$.

An element $a$ of $S$ is said to be regular (completely regular)
$\cite{ke92}$ if there $x \in S$ such that $a\leq axa \; (a\leq
a^2xa^2)$. $S$ is called a regular (completely regular) ordered
semigroup if every element of $S$ is regular (completely regular).
Note that $S$ is regular (completely regular) if and only if $a\in
(aSa]\; (a\in (a^2Sa^2])$ for all $a\in S$. The set of regular and
completely regular elements in an ordered semigroup $S$ we denote
by $Reg_{\leq}(S)\;\textrm{and} \; Gr_{\leq}(S)$ respectively. $S$
is called intra-regular if for every $a \in S, \;a \in (S a^2 S]$.
The set of all intra-regular elements in ordered semigroup $S$ we
denote by $Intra_\leq (S)$.

 An element $b \in S$ is ordered inverse $\cite{HJ}$ of $a$ if $a \leq aba$ and
$b \leq bab$.  The set of all ordered inverses of an element $a
\in S$ is denoted by $V_\leq(a)$. Throughout the paper $a', a''$
are the ordered inverses of $a$ unless otherwise stated.

An element $e\in S$ is said to be ordered idempotent $\cite{bh1}$
if $e\leq e^2$. The set of all ordered idempotents of $S$ is
denoted by $E_\leq (S)$.

Cao and Xinzhai $\cite{Cao 2000}$ studied ordered semigroups which
are nil-extensions of  $t$-simple ordered semigroups. Sadhya and
Hansda $\cite{SH}$ studied these semigroups under the name of
$\pi$-$t$-simple ordered semigroups. In fact class of these
ordered semigroups are natural generalization of $\pi$-groups.
Jamadar and Hansda $\cite{JH}$ studied ordered semigroup in which
any two inverses of an element are $\rc$-related. Class of these
ordered semigroups are natural generalization of class of right
inverse semigroups. Hansda and Jamadar $\cite{JH}$ named these
ordered semigroups as right inverse ordered semigroups and studied
their different aspects. In this paper an extensive study of
$\pi$-$t$-simple and right inverse ordered semigroups has been
done. Our approach allows one to see an interplay between these
two ordered semigroups.

A nonempty subset $A$ of $S$ is  called a left (right) ideal
$\cite{Ke2006}$ of $S$, if $SA \subseteq A \;( A S \subseteq A)$
and $(A]= A$. A nonempty subset $A$ is called a (two-sided) ideal
of $S$ if it is both a left and a right ideal of $S$.

The principal $\cite{Ke2006}$ left ideal, right ideal, ideal and
bi-ideal $\cite{ke92}$ generated by $a \in S$ are denoted by
$L(a), \;R(a), \;I(a)$and $B(a)$ respectively and defined by
$$L(a)= (a \cup Sa], \;R(a)= (a\cup aS], \;I(a)= (a\cup Sa \cup aS
\cup SaS] \;  and \;B(a)=(a \cup aSa].$$

Kehayopulu \cite{Ke2006} defined Greens relations $\lc, \;\rc,
\;\jc \;\textrm{and} \;\hc$ on an ordered semigroup $S$ as
follows:

$$ a \lc b   \; if   \;L(a)= L(b), \; a \rc b   \; if   \;R(a)=
R(b),\;
 a \jc b   \; if   \;I(a)= I(b) \;\textrm{and} \;\hc= \;\lc \cap \;\rc.$$

These four relations  are equivalence relations on $S$.

An ordered semigroup $S$ is called $\pi$-regular (resp. completely
$\pi$-regular) $\cite{Cao 2000}$ if for every $ a\in S$ there is
$m \in \mathbb{N}$ such that $a^{m} \in (a^{m}Sa^{m}] \;
(\textrm{resp}. \;a^{m} \in (a^{2m}Sa^{2m}]).$ The set of
$\pi$-regular elements in an ordered semigroup $S$ is denoted by
$\pi Reg_{\leq}(S)$. An ordered semigroup $S$ is called left
 $\pi$-regular (resp. right $\pi$-regular) $\cite{Cao
2000}$ if for every $ a\in S$ there is $m \in \mathbb{N}$ such
that $a^{m} \in (Sa^{2m}] \; (\textrm{resp}. \;a^{m} \in
(a^{2m}S]).$

An ordered semigroup $S$ is said to be weakly commutative
$\cite{ke 2008}$ if for all $a, b \in S$, $(ab)^{n} \in (bSa] \;
\textrm{for some} \; n \in \mathbb{N}$. $S$ is right (left) weakly
commutative $\cite{ke 2008}$ if for every $a, b \in S$ there
exists $n \in \mathbb{N}$ such that $(ab)^n\in (Sa]$ ($(ab)^n\in
(bS]$).

A band is a semigroup in which every element is an idempotent. A
commutative band is called a semilattice. An ordered semigroup $S$
is said to be Archimedean  $\cite{ke 2008}$ if for every $a, b \in
S$ there exists $n \in \mathbb{N}$ such that $a^{n} \in (SbS]$. An
ordered semigroup $S$ is said to be left (right) Archimedean
$\cite{ke 2008}$ if for every $a,b \in S$ there exists $n \in
\mathbb{N}$ such that $a^{n} \in (Sb] \; (a^{n} \in (bS])$.

In an ordered semigroup $S$, an equivalence relation $\rho$ is
called left (right) congruence if for $a, b, c \in S \;a \;\rho \;b
\; \textrm{implies} \;ca \;\rho \;cb \;(ac \;\rho \;bc)$. $\rho$ is
congruence if it is both left and right congruence. A congruence
$\rho$ on $S$ is called semilattice if for all $a, b \in S \;a
\;\rho \;a^{2} \;\textrm{and} \;ab \;\rho \;ba$. A semilattice
congruence $\rho$ on $S$ is called complete if $a \leq b$ implies
$(a,ab)\in \rho$. The ordered semigroup $S$  is called complete
semilattice of subsemigroup of type $\tau$ if there exists a
complete semilattice congruence $\rho $ such that $(x)_{\rho}$ is a
type $\tau$ subsemigroup of $S$. Equivalently: There exists a
semilattice $Y$ and a family of subsemigroups $\{S\}_{\alpha \in Y}$
of type $\tau$ of $S$ such that:
\begin{enumerate}
\item \vspace{-.4cm} $S_{\alpha}\cap S_{\beta}= \;\phi$ for any
$\alpha, \;\beta \in \;Y \;with \; \alpha \neq \beta,$ \item
\vspace{-.4cm} $S=\bigcup _{\alpha \;\in \;Y} \;S_{\alpha},$ \item
\vspace{-.4cm} $S_{\alpha}S_{\beta} \;\subseteq \;S_{\alpha
\;\beta}$ for any $\alpha, \;\beta \in \;Y,$ \item \vspace{-.4cm}
$S_{\beta}\cap (S_{\alpha}]\neq \phi$ implies $\beta \;\preceq
\;\alpha,$ where $\preceq$ is the order of the semilattice $Y$
defined by \\$\preceq:=\{(\alpha,\;\beta)\;\mid
\;\alpha=\alpha\;\beta\;(\beta\;\alpha)\}$ $\cite{ke 2008}$.
\end{enumerate}

Let $S$ be an ordered semigroup and $\rho$ be an equivalence
relation on $S$. Following Hansda and Jamadar $\cite{HJ}$, an
element $a\in S$ of type $\tau$ is said to be $\rho$-unique
element in $S$ if for every other element $b\in S$ of type $\tau$
we have $a\rho b$.

For the sake of convenience of general reader we state following
results.

\begin{Corollary}\cite{Cao 2000}\label{1}
The following conditions are equivalent on a po-semigroup $S$:
\begin{enumerate}
\item\vspace{-.4cm}
$S$ is a nil-extension of a left simple
 $\pi$-regular po-semigroup;

\item\vspace{-.4cm}

($\forall a,b \in S$)( $\exists m \in
\mathbb{Z^+}$)  $a^m \in (a^mSb^n]$ for every $n \in
\mathbb{Z^+}$;

\item\vspace{-.4cm}

($\forall a,b \in S$)( $\exists m \in \mathbb{Z^+}$)  $a^m \in
(a^mSb]$;

\item\vspace{-.4cm}

($\forall a,b \in S$)( $\exists m \in \mathbb{Z^+}$)  $a^m \in
(a^mSb^{m}]$;

 \item\vspace{-.4cm}
 $S$ is a l-Archimedean
$\pi$-regular po-semigroup.

\end{enumerate}
\end{Corollary}

Let $S$ be a $\pi$-regular ordered semigroup. Due to Sadhya and
Hansda \cite{SH1}, the following equivalence relations $\lc^*$,
$\rc^*$, $\jc^*$ and $\hc^*$ are defined by:

\begin{enumerate}
\item\vspace{-.4cm}
 $a\lc^* b\Leftrightarrow a^m\lc b^n$
\item\vspace{-.4cm}
$a\rc^* b\Leftrightarrow a^m\rc b^n$
\item\vspace{-.4cm}
$a\jc^* b\Leftrightarrow a^m\jc b^n$
\item\vspace{-.4cm}
$\hc^*= \lc^*\cap \rc^*$
\end{enumerate}

where $a,b\in S$ and $m, n$ are the smallest positive integers
such that $a^m, b^n\in Reg_\leq (S)$.

\begin{Lemma}\cite{AJ}\label{3}
Let $S$ be a $\pi$-regular ordered semigroup and $a\in S$. Then
there is $m \in \mathbb{N}$ such that $(Sa^m]$ is generated by an
ordered idempotent.
\end{Lemma}

\begin{Theorem}\cite{JH}\label{51}
The following conditions are equivalent on a regular ordered
semigroup $S$.
\begin{enumerate}
\item\vspace{-.4cm} $S$ is right inverse; \item\vspace{-.4cm} for
every $a\in S$ and $a', a'' \in V_{\leq}(a)$, $a'\rc a''$;
\item\vspace{-.4cm} for every $e,f\in E_{\leq}(S)$, $ef\in
(fSeSf]$; \item\vspace{-.4cm} for every $e,f\in E_{\leq}(S)$,
$(eS] \cap (fS]=(efS]$; \item\vspace{-.4cm} for $e\in E_{\leq}(S)$
and $x\in (Se]$\textrm{ implies} $x'\in (eS]$, where $x\in S$ and
$x'\in V_{\leq}(x)$.
\end{enumerate}
\end{Theorem}

\section{Left $\pi$-$t$-simple ordered semigroup}

This section deals with the characterizations of left
$\pi$-$t$-simple ordered semigroups.

\begin{Theorem}\label{7}
A regular ordered semigroup $S$ is a left simple and $\pi$-regular
ordered semigroup if and only if $e \lc^* f$ for all $e, f \in
E_{\leq}(S)$.
\end{Theorem}
\begin{proof}
This follows immediately.
\end{proof}

\begin{Definition}
An ordered semigroup $S$ is said to be left $\pi$-$t$-simple
ordered semigroup if there exists a left simple and $\pi$-regular
ordered subsemigroup $H$ of $S$ with the property that for every
$a \in S$ there exists $m \in \mathbb{N}$ such that $a^{m} \in H$.
\end{Definition}
Some characterizations of a  left $\pi$-$t$-simple ordered
semigroup have been given in the following theorem.

\begin{Theorem}\label{2}
Let $S$ be an ordered semigroup. Then the following conditions are
equivalent:
\begin{enumerate}
\item\vspace{-.4cm}

$S$ is a left $\pi$-$t$-simple ordered semigroup;

\item\vspace{-.4cm}

$S$ is $\pi$-regular and contains an $\lc^*$-unique ordered
idempotent;

\item\vspace{-.4cm}

$S$ is $\pi$-regular and $a\lc^* b$ for every $a, b\in S$;

\item\vspace{-.4cm}

for every $a,b \in S$ there exists $m \in \mathbb{N}$ such that
$a^m \in (a^mSb]$;

\item\vspace{-.4cm}

for every $a,b \in S$ there exists $m \in \mathbb{N}$ such that
$a^m \in (a^mSb^n]$ for every $n \in \mathbb{N}$;

\item\vspace{-.4cm}

for every $a,b \in S$ there exists $m \in \mathbb{N}$ such that
$a^m \in (a^mSb^m]$;

\item\vspace{-.4cm}

$S$ is $\pi$-regular and $l$-Archimedean;

\item\vspace{-.4cm}

$S$ is a nil-extension of a left simple $\pi$-regular ordered
semigroup.

\end{enumerate}
\end{Theorem}
\begin{proof}
$(1)\Rightarrow(2)$: Let $a \in S$. Then there exist a left simple
and $\pi$-regular ordered subsemigroup $H$ of $S$ and $m \in
\mathbb{N}$ such that $a^m \in H$.  Since $H$ is a $\pi$-regular
ordered subsemigroup, so for $a^{m} \in H$ there exist $r \in
\mathbb{N}$ and $z\in H$ such that $(a^m)^r \leq (a^m)^rz(a^m)^r$,
that is $a^k\leq a^kza^k$, where $k= mr$. Hence $S$ is
$\pi$-regular.

Let $ e, f \in E_{\leq}(S)$. Then there are $m, n \in \mathbb{N}$
such that $e^{m}, f^{n} \in H$. Since $H$ is a left simple ordered
subsemigroup, $e^{m} \leq xf^{n}, \; f^{n} \leq ye^{m}$ for some
$x, y \in H$. Thus $e \leq e^{m} \leq xf^{n} =x (f^{n-1}f) $,
which implies that $e \leq s_{1}f$ for some $s_{1}\in S$.
Similarly we obtain that $ f \leq s_{2}e $ for some $s_{2} \in S
$. So $e \lc f$. Hence $e\lc^* f$.

$(2)\Rightarrow(3)$: Let $a, b \in S$. Since $S$ is $\pi$-regular,
let $m, n$ be the smallest positive integers such that $a^m,
b^n\in Reg_\leq(S)$. Then $a^m \leq a^m sa^m \;\textrm{and} \;
b^n\leq b^ntb^n$ for some $s,t \in S$. Since $a^ms,
b^nt,sa^m,tb^n\in E_\leq(S)$ and ordered idempotents of $S$ are
$\lc^*$-unique, there exists $x \in S$ such that $sa^m \leq xt
b^n$. Hence $a^m \leq a^msa^m \leq a^mxtb^n$ and thus $a^m \leq
s_1b^n$ where $s_1= a^mxt$. Similarly $b^n\leq s_2a^m$ for some
$s_2\in S$. Therefore $a^m\lc b^n$ and so $a\lc^* b$. Hence $S$ is
$\pi$-regular and $a\lc^* b$ for every $a, b\in S$.

$(3)\Rightarrow(4)$: This is obvious.

$(4)\Rightarrow(5)$: Let $a,b \in S$. Then by given condition
there exists $m \in \mathbb{N}$ and $x\in S$ such that $a^m \leq
a^mxb \leq a^m(xb)^2 \leq a^m(xb)^3  \leq \cdots$. Thus $a^m\leq
a^m(xb)^n$ for all $n \in \mathbb{N}$. Also for $xb, b^n\in S$,
the condition (4) yields that $(xb)^k \leq (xb)^kyb^n$, for some
$k \in \mathbb{N}$ and $y\in S$. Thus $a^m\leq a^m(xb)^k\leq
a^m(xb)^kyb^n$, that is, $a^m\in (a^mSb^n]$ for all $n \in
\mathbb{N}$.

$(5)\Rightarrow(6)$: This is obvious.

$(6)\Rightarrow(7)$ and $(7)\Rightarrow(8)$: These follow from
Corollary $\ref{1}$.

$(8)\Rightarrow(1)$: Suppose $S$ is a nil-extension of a left
simple $\pi$-regular ordered semigroup $K$. Clearly $K$ is a
subsemigroup of $S$ and for $a\in S$ there exists $m \in
\mathbb{N}$ such that $a^m\in K$. Therefore $S$ is a left
$\pi$-$t$-simple ordered semigroup.
\end{proof}

\begin{Theorem}
An ordered semigroup  $S$ is a left $\pi$-$t$-simple ordered
semigroup if  $S$ is right weakly commutative and $r$-Archimedean
with an $\lc^*$-unique ordered idempotent.
\end{Theorem}
\begin{proof}
First suppose that $S$ is a right weakly commutative and
$r$-Archimedean ordered semigroup with an $\lc^*$-unique ordered
idempotent $e$. Then $S$ is an Archimedean ordered semigroup. Now
$e \leq e(e^{2})e$. Therefore $e \in Intra(S)$ so that $Intra(S)
\neq \phi $. Then $S$ is a nil-extension of a simple ordered
semigroup $K$ by Theorem 3.5 of $\cite{Cao 2000}$. Let $a \in K$.
Since $K$ is simple, $a \leq xa^{3}y$ for some $x, y \in K$, which
implies that $a \leq (xa)a(ay) \leq (xa)^{2}a(ay)^{2} \leq \ldots
\leq (xa)^{r}a(ay)^{r}$ for every $r \in \mathbb{N}$. Since $S$ is
right weakly commutative, there exist $n \in \mathbb{N}$ such that
$(ay)^{n} \in (Sa]$.  Also since $S$ is $r$-Archimedean, there
exists $m \in \mathbb{N}$ such that $(xa)^m\in (aS]$. Thus there
exist $z, w \in S$ such that $(xa)^{m} \leq az$ and $(ay)^{n} \leq
wa$. Hence $a \leq (xa)^{m}a(ay)^{n} \leq azawa$. Since $K$ is an
ideal of $S$, it follows that $zaw \in K$ and thus $a \in (aKa]$.
So $K$ is a regular ordered semigroup and hence $K$ is
$\pi$-regular ordered semigroup.

Let $e, f \in E_{\leq}(K)$. Then by given condition $e \lc^* f$ in
$S$. Since $K$ is an ideal of $S$, it is evident that $e \lc^* f$
in $K$ also. Therefore $K$ is a left simple and $\pi$-regular
ordered semigroup and so $S$ is a nil-extension of a left simple
$\pi$-regular ordered semigroup $K$. Hence $S$ is a left
$\pi$-$t$-simple ordered semigroup, by Theorem \ref{2}.

\end{proof}

\begin{Theorem}\label{4}
Let $S$ be an ordered semigroup. Then the following conditions are
equivalent:

\begin{itemize}
\item [(1)]

$S$ is a  semilattice of left $\pi$-$t$-simple ordered semigroups;

\item [(2)]

$S$ is  $\pi$-regular and for every $a, b \in S$, $ab\lc^* ba$;

\item [(3)]

$S$ is $\pi$-regular and right weakly commutative;

\item[(4)]

for all $a, b\in S$ there exists $m \in \mathbb{N}$ such that
$(ab)^m\in ((ab)^mS(ba)^{m+1}]$;

\item[(5)]

$S$ is a  semilattice of nil-extension of a left simple and
$\pi$-regular ordered semigroup.

\end{itemize}

\begin{proof}
$(1) \Rightarrow (2) $: Let $S$ be a  semilattice $Y$ of left
$\pi$-$t$-simple ordered semigroups $ \{S_{\alpha}\}_{\alpha \in
Y}$. By Theorem $\ref{2}$, for each  $\alpha \in Y$, $S_{\alpha}$
is $\pi$-regular and thus  so is $S$. Let $a, b \in S$. Then $a
\in S_{\alpha}$ and $b \in S_{\beta}$ for some $\alpha , \beta \in
Y$. Thus $ab \in S_{\alpha \beta}$ and $ba, b^ra^r \in S_{\beta
\alpha}$ for all $r \in \mathbb{N}$. Since $Y$ is a semilattice,
so $S_{\alpha \beta}= S_{\beta \alpha}$. Therefore $ba, b^ra^r \in
S_{\alpha \beta}$. Now $(ab)^n\leq (ab)^nz(ab)^n$ and $(ba)^m\leq
(ba)^mt(ba)^m$ for some $z, t \in S_{\alpha \beta}$, where $n, m$
are the smallest positive integers such that $(ab)^n, (ba)^m \in
Reg_\leq (S)$. Clearly $(ab)^n\lc z(ab)^n$ and $(ba)^m\lc t(ba)^m$
and $z(ab)^n, t(ba)^m\in E_\leq (S_{\alpha \beta})$. Also from
Theorem \ref{2}, $S_{\alpha \beta}$ contains an $\lc^*$-unique
ordered idempotent. That is any two ordered idempotents in
$S_{\alpha \beta}$ are $\lc^*$-related. This yields that
$(ab)^n\lc z(ab)^n\lc t(ba)^m\lc (ba)^m$. Hence $ab\lc^* ba$.

$(2) \Rightarrow (3) $:  Let $a, b\in S$. Then by condition (2) it
follows that $ab\lc^* ba$. Then $(ab)^n\lc (ba)^m$, where $n, m$
are the smallest positive integers such that $(ab)^n, (ba)^m \in
Reg_\leq (S)$. Then we have $(ab)^n\leq (ab)^ns(ab)^n$ for some
$s\in S$. Also $(ab)^n\lc (ba)^m$ gives that $(ab)^n\leq t(ba)^m$
for some $t\in S$. This yields that $(ab)^n \leq (ab)^nst(ba)^m$,
that is $(ab)^n\in (Sa]$. Hence $S$ is a $\pi$-regular  and right
weakly commutative ordered semigroup.

$(3) \Rightarrow (4) $: This is obvious.

$(4) \Rightarrow (5) $: By Corollary $4.5$  of $\cite{ke 2008}$ it
is evident that $S$ is a semilattice $Y$ of semigroups
$\{S_\alpha\}_{\alpha \in Y}$ and for each $\alpha \in Y$
$\{S_\alpha\}$  is a $l$-Archimedean ordered semigroup . Also by
condition (4) each $\{S_\alpha\}_{\alpha \in Y}$ is a
$\pi$-regular ordered semigroup. Hence each $\{S_\alpha\}_{\alpha
\in Y}$ is a nil-extension of a left simple and $\pi$-regular
ordered semigroup $\{K_\alpha\}_{\alpha \in Y}$ for every $\alpha
\in Y$, by Theorem $\ref{2}$.

$(5) \Rightarrow (1)$: This is obvious.
\end{proof}

\end{Theorem}

Right $\pi$-$t$-simple ordered semigroup can be characterized
dually.

\section{Right $\pi$-inverse ordered semigroup}

This section is devoted to a detailed exposition of right
$\pi$-inverse ordered semigroups.

\begin{Definition}
A $\pi$-regular ordered semigroup $S$ is called right
$\pi$-inverse if every $a\in S$ there is $m \in \mathbb{N}$ such
that $(Sa^m]$ is generated by an $\rc$-unique ordered idempotent
of $S$.
\end{Definition}

\begin{Theorem}\label{5}
The following conditions are equivalent on a $\pi$-regular ordered
semigroup $S$.
\begin{enumerate}
\item\vspace{-.4cm}
$S$ is right $\pi$-inverse;
\item\vspace{-.4cm}
for $a\in S$ there is $m \in \mathbb{N}$ such
that $a', a'' \in V_{\leq}(a^m)$ \textrm{ implies}  $a'\rc a''$;
\item\vspace{-.4cm}
for $e,f\in E_{\leq}(S)$, there is $n \in
\mathbb{N}$ such that $(ef)^n \in (fSf]$;
\item\vspace{-.4cm}

for $e,f\in E_{\leq}(S)$, $((ef)^nS]\subseteq (eS] \cap (fS]$ for
some $n \in \mathbb{N}$;
\item\vspace{-.4cm}

for $e\in E_{\leq}(S)$ there is $m \in \mathbb{N}$ such that
$x^m\in (Se]$\textrm{ implies} $x'\in (eS]$, where $x\in S$ and
$x'\in V_{\leq}(x^m)$.
\end{enumerate}
\end{Theorem}

\begin{proof}
$(1)\Rightarrow(2)$: Let $S$ be a right $\pi$-inverse semigroup
and $a\in S$. Clearly $S$ is $\pi$-regular. Consider  $a',a''\in
V_{\leq}(a^m)$ for some $m \in \mathbb{N}$. Then $a'a^m,a''a^m\in
E_{\leq}(S)$. Say $L=(Sa^m]$ and let $x\in (Sa^m]$. Then there is
$s\in S$ such that $x\leq sa^m$. Thus $x\leq sa^ma'a^m$, which
implies that $x\in (Sa'a^m]$ and so $(Sa^m]\subseteq (Sa'a^m]$.
Also $(Sa'a^m]\subseteq (Sa^m]$. Hence $(Sa^m]=(Sa'a^m]$.
Similarly $(Sa^m]=(Sa''a^m]$. Hence $(Sa^m]=(Sa'a^m]=(Sa''a^m]$.
Since $S$ is right $\pi$-inverse, we have $a'a^m\rc a''a^m$. Now
$a''\leq a''a^ma'' \leq a'a^mz_1a''$ for some $z_1\in S$. Hence
$a''\leq a't_1$, where $t_1=a^mz_1a''$. Also $a'\leq a'a^ma'\leq
a''a^mz_2a'$ for some $z_2\in S$. Hence $a' \leq a''t_2$, where
$t_2=a^mz_2a'$. Hence $a'\rc a''$.

$(2)\Rightarrow (3)$: Let $e,f\in E_{\leq}(S)$. Since $S$ is
$\pi$-regular, $V_{\leq}(ef)^n \neq\phi$ for some $n \in
\mathbb{N}$. Let $x\in V_{\leq}(ef)^n$.
We consider the following cases:\\
Case 1: If $n=1$ then $ef \in (fSf]$ holds, by Theorem \ref{51}.\\
Case 2: If $n>1$ then $x\leq x(ef)^nx$ implies that $fxe\leq
fxe(ef)^nfxe$. Also $(ef)^n\leq (ef)^nx(ef)^n$ implies that
$(ef)^n\leq (ef)^n(fxe)(ef)^n$. Thus $(ef)^n\in V_\leq(fxe)$. Now
$x\leq x(ef)^nx= xe(fe)^{n-1}fx$ so that $fxe\leq
fxe(fe)^{n-1}fxe\leq fxe(fe)^{n-1}fxe(fe)^{n-1}fxe$ and
$(fe)^{n-1}fxe(fe)^{n-1}\leq
(fe)^{n-1}fxe(fe)^{n-1}fxe(fe)^{n-1}\\\leq
(fe)^{n-1}fxe(fe)^{n-1}fxe(fe)^{n-1}fxe(fe)^{n-1}$. This gives
$(fe)^{n-1}fxe(fe)^{n-1}\in V_\leq(fxe)$.\\ Thus $(ef)^n,
(fe)^{n-1}fxe(fe)^{n-1}\in V_\leq(fxe)$. Then by condition (2), we
have $(ef)^n\rc (fe)^{n-1}fxe(fe)^{n-1}$. So $(ef)^n\leq
(ef)^nx(ef)^n \leq (fe)^{n-1}fxe(fe)^{n-1}ux(ef)^n$ for some $u\in
S$.  Hence $(ef)^n\in (fSf]$.

$(3)\Rightarrow (4)$: Let $e,f \in E_\leq(S)$. Then by condition
(3), there is $n \in \mathbb{N}$ such that $(ef)^n\in (fSf]$. Let
$y\in ((ef)^nS]$. Then $y\in (eS]$ and $y\leq (ef)^nq$ for some
$q\in S$.  Then there is $s_6\in S$ such that $y\leq fs_6fq$. Thus
$y\in (fS]$ and so $y\in (eS]\cap (fS]$. Therefore $
((ef)^nS]\subseteq (eS]\cap (fS]$.

$(4)\Rightarrow (5)$: Let $x\in S$ and $e\in E_\leq(S)$. Let
$x'\in V_{\leq}(x^m)$ for some $m \in \mathbb{N}$. Then $x^m\leq
x^mx'x^m$ and $x'\leq x'x^mx'$. Let $x^m\in (Se]$ . Then there is
$s\in S$ such that $x^m\leq se$. Now $x'se\leq x'x^mx'se\leq
x'sex'se$ so that $x'se\in E_\leq(S)$. Then for $ x'se,e\in
E_{\leq}(S) $ we have that $ ((x'see)^nS]\subseteq (eS]\cap
(x'seS]$, for some $n \in \mathbb{N}$ by condition (4). Also from
$x'\leq x'x^mx'$ we have $x'\leq x'sex'\leq x'seex'\leq
x'seex'seex'\leq \cdots$. Therefore $x'\in ((x'See)^rS]$ for all
$r \in \mathbb{N}$. So $x'\in ((x'See)^nS]$ and hence $x'\in
(eS]$.

$(5)\Rightarrow(1)$:  Let $a\in S$. Since $S$ is $\pi$-regular,
there is $m \in \mathbb{N}$ such that $I=(Sa^m]=(Se]$ for some
$e\in E_\leq(S)$, by Lemma $\ref{3}$. If possible let $I= (Sf]$
for some $ f \in E_\leq(S)$ other than $e$. Now $e^m\leq e^me$, so
$e^m\in (Se]=(Sf]$. Now $e^m\in V_\leq(e)$. So from the give
condition $e^m\in (fS]$. Similarly $f^m\in (eS]$ and thus $e^m\rc
f^m$ implies $e\rc f$. Hence $S$ is a right $\pi$-inverse ordered
semigroup.

\end{proof}
\begin{Corollary}
An ordered semigroup $S$ is  $\pi$-inverse  if and only if $S$ is
both a left and a right $\pi$-inverse ordered semigroup.
\end{Corollary}

\begin{Theorem}\label{6}
A $\pi$-regular ordered semigroup $S$ is  right $\pi$-inverse if
and only if for any two idempotents $e,f\in E_\leq(S)$, $e\lc^* f$
implies $e\rc^* f$.
\begin{proof}
First suppose that  $S$ is a right $\pi$-inverse ordered
semigroup. Let $e,f\in E_\leq (S)$ be such that $ e\lc^* f$.
Clearly $e\lc f$. Then $e\leq xf$ and $f\leq ye$ for some $x,y\in
S$. Now $e\leq xf $ implies $e\leq exf$, and so $e\leq ee\leq
exfe$, which implies that $exf\leq exfexf$. So $exf\in
E_{\leq}(S)$. Similarly $f\leq fye$ and $fye\in E_{\leq}(S)$. Now
$e\leq exf\leq exff \leq (exf)(fye)$. Since $exf, fye\in
E_\leq(S)$ and $S$ is right $\pi$-inverse $(exffye)^m \leq
fyesexf$, for some  $m \in \mathbb{N}$ and $s\in S$. Thus $e\leq
e^m$ implies that $e \leq(exffye)^m$ and therefore $e\leq
fyesexf$. Hence $e\leq ft_1$, where $t_1= yesexf$. Similarly
$f\leq et_2$ for some $t_2\in S$. Therefore $e\rc f$. Hence
$e\rc^* f$.

Conversely assume that the given condition holds in $S$. Let $a\in
S$ and consider the left ideal  $L=(Sa^m]$ for some $m \in
\mathbb{N}$. Then it is evident that there is  $e\in E_\leq(S)$
such that $L =(Se]$. If possible let $L= (Sf]$ for some $f\in
E_\leq(S)$. Then $e\lc f$ which implies that $e\lc^* f$. So by the
given condition $e\rc^*f$ and hence $e\rc f$. Hence $S$ is a right
$\pi$-inverse ordered semigroup.
\end{proof}
\end{Theorem}

\begin{Corollary}
An ordered semigroup which is both right $\pi$-inverse and left
$\pi$-$t$-simple is a $\pi$-$t$-simple ordered semigroup.
\end{Corollary}
\begin{proof}
This follows from Theorem $\ref{2}$ and Theorem $\ref{6}$.
\end{proof}

\begin{Lemma}\label{7}
Let $S$ be a right $\pi$-inverse ordered semigroup. Then $a\lc^*
b$ implies $a'a^m\rc^* b'b^n$, for some $a,b\in S$ and $m, n \in
\mathbb{N}$, where $a'\in V_\leq(a^m)$, $b'\in V_\leq(b^n)$.
\end{Lemma}
\begin{proof}
Let $a,b\in S$ such that $a\lc^* b$. Then $a^m\lc b^n$, where $m,
n$ are the smallest positive integers such that $a^m, b^n\in
Reg_\leq(S)$. Let $a'\in V_{\leq}(a^m)$, $b'\in V_{\leq}(b^n)$.
Since $a^m\leq a^ma'a^m$ and $a'a^m\leq a'a^ma'a^m$, we have
$a^m\lc a'a^m$, which implies that $b^n\lc a'a^m$. Also $b^n\lc
b'b^n$. Hence $a'a^m\lc b'b^n$. Since $a'a^m, b'b^n\in
E_{\leq}(S)$ therefore $a'a^m\lc^* b'b^n$. Since $S$ is right
$\pi$-inverse we have $a'a^m\rc^* b'b^n$, by Theorem $\ref{6}$.

\end{proof}

Left $\pi$-inverse ordered semigroup can be characterized dually.

In the following theorem we show that $S$ is a  semilattice of
right $\pi$-$t$-simple ordered semigroups if and only if $S$ is
right $\pi$-inverse and $\rc^*$ is a congruence on $S$.
\begin{Theorem}\label{8}
Let $S$ be a right $\pi$-inverse ordered semigroup. Then following
conditions are equivalent:
\begin{enumerate}
\item\vspace{-.4cm}
$\rc^*$ is a congruence on $S$;
\item\vspace{-.4cm}
$\lc^* \subseteq\rc^*$;
\item\vspace{-.4cm}

$S$ is a semilattice of right $\pi$-$t$-simple ordered semigroups;
\item\vspace{-.4cm}
$\rc^*$ is a semilattice congruence on $S$.
\end{enumerate}
\end{Theorem}
\begin{proof}
$(1)\Rightarrow (2)$. Let $a, b\in S$ such that $a\lc^* b$. Since
$S$ is right $\pi$-inverse, so $S$ is $\pi$-regular and
$a'a^m\rc^* b'b^n$, by Lemma $\ref{7}$. For $a\in S$ there is
$a'\in V_\leq (a^m)$, where $m$ is the smallest positive integer
such that $a^m\in Reg_\leq(S)$. Now $a^m\leq a^ma'a^m$ and $a'\leq
a'a^ma'$. Denote $a^ma'= e$ and $a'a^m= f$. Therefore $a^m\rc e$
and $a'\rc f$. So $a\rc^*e$ and $a'\rc^*f$. Also $a\rc^*e$,
$a\rc^*e$, $\cdots$, $m$-times yields that $a^m\rc^*e^m\rc^* e$.
Then congruenceness of $\rc^*$ on $S$ gives that $ a^ma'\rc^* ef$,
that is $e\rc^* ef$. Similarly $f\rc^* fe$. So there is $z\in S$
such that $f\leq (fe)^nz$, where $n$ is the least positive integer
such that $(fe)^n\in Reg_\leq(S)$. Since $S$ is right
$\pi$-inverse, there is $x\in S$ such that $(fe)^n\leq exe$, by
Theorem \ref{5}. Hence $f\leq (fe)^nz\leq exez= et_1$, where $t_1=
xez$. Similarly $e\leq ft_2$, for some $t_2\in S$. Therefore $f\rc
e$. So $f\rc^*e$. Hence $a\rc^* e\rc^* ef\rc^* fe\rc^* f\rc^* a'$.
Therefore $a\rc^* a'a^m\rc^* b'b^n\rc^* b$. Hence $a\rc^* b$. So
$\lc^*\subseteq \rc^*$.

$(2)\Rightarrow (3)$: Let $a\in S$. Since $S$ is $\pi$-regular,
let $a'\in V_\leq(a^m)$, where $m$ is the smallest positive
integer such that $a^m\in Reg_\leq(S)$. Clearly $a\lc^* a'a^m$ and
so $a\rc^* a'a^m$. Thus $a^m\rc a'a^m$. Therefore $a^m\leq
a^ma'a^m\leq a^{2m}t$ for some $t\in S$. Hence $S$ is a right
$\pi$-regular ordered semigroup.

Let $n$ be the smallest positive integer such that $(ab)^n\in
Reg_\leq(S)$. Then there is $s\in S$ such that $(ab)^n\leq
(ab)^ns(ab)^n\leq (ab)^ns(ab)^ns(ab)^n\leq
(ab)^nsa(ba)^{n-1}bs(ab)^n$. Also $(ba)^{n-1}bs(ab)^n\leq
(ba)^{n-1}bs(ab)^ns(ab)^n\leq (ba)^{n-1}bs(ab)^ns(ab)^ns(ab)^n\leq
(ba)^{n-1}bs(ab)^nsa(ba)^{n-1}bs(ab)^n$ which implies that
$(ba)^{n-1}bs(ab)^n\in Reg_\leq(S)$ and hence $ab\lc^*
(ba)^{n-1}bs(ab)^n$. Therefore $ab\rc^* (ba)^{n-1}bs(ab)^n$, by
condition (2). Hence $(ab)^n\leq (ba)^{n-1}bs(ab)^nt_1$ for some
$t_1\in S$. Then $S$ is a left weakly commutative and
$\pi$-regular ordered semigroup. Hence $S$ is a semilattice of
right $\pi$-$t$-simple ordered semigroups, by Theorem $\ref{4}$.

$(3)\Rightarrow (4)$: Let $a, b, c\in S$ such that $a\rc^*b$. Then
$a^m\rc b^n$, where $m, n$ are the smallest positive integers such
that $a^m, b^n\in Reg_\leq(S)$.  Since $S$ is a semilattice $Y$ of
right $\pi$-$t$-simple ordered semigroups $ \{S_{\alpha}\}_{\alpha
\in Y}$, we have that $ab\lc^* b^ra^r$, by Theorem $\ref{4}$. Also
each $ \{S_{\alpha}\}_{\alpha \in Y}$ is $l$-Archimedean and
$\pi$-regular and hence left $\pi$-regular. So $S$  is left
$\pi$-regular and  thus  $a\lc^* a^2$. Hence

$$a\rc^*a^2, ab\rc^*ba \;\textrm{and} \;ab\rc^* b^ra^r \rc^* a^rb^r
\;\textrm{for} \;\textrm{all} \;r \in
\mathbb{N}.~~~~~~~~~~~~~~~~~~~~~~~~~~~~~~~~~~~~~~~~~(*)$$

Now  $a\rc^*a^2$ implies $a^m\rc a^{2m}$, $a^{2m}\rc a^{4m}$,
$\cdots$ . So $a^{m}\rc a^{2mn}$. Similarly $b^{n}\rc b^{2mn}$.
Therefore $a^{2mn}\rc a^m\rc b^n\rc b^{2mn}$. Since $\rc$ is a
left congruence so $c^{2mn}a^{2mn}\rc c^{2mn}b^{2mn}$. Now we
shall show that $\rc\subseteq \rc^*$. Let $u, v\in S$ such that
$u\rc v$. Then $u\leq vx$ for some $x\in S$. Therefore $u^s\leq
(vx)^s$, where $s$ is the smallest positive integer such that
$u^s\in Reg_\leq(S)$. Let $t$ is the smallest positive integer
such that $(vx)^t \in Reg_\leq(S)$.

Case(1): If $s\geq t$ then $u^s\leq (vx)^s= (vx)^t(vx)^{s-t}\leq
(v^rx^r)^{u_1}s_1(vx)^{s-t}$ for some $s_1\in S$ and for all $r
\in \mathbb{N}$, where $u_1$ is the smallest positive integer such
that $(v^rx^r)^{u_1}\in Reg_\leq(S)$, by (*).

Case(2): If $s< t$ then  there exists positive integer $n$ such
that $ns> t$. Now $u^s\leq u^{2s}z_2\leq u^{3s}z_3\leq \cdots \leq
u^{ns}z_n \leq (vx)^{ns}z_n$. So there exists $s_2\in S$ such that
$u^s\leq (v^rx^r)^{u_1}s_2$. So in either case $u^s\leq v^rw$ for
some $w\in S$ and for all $r \in \mathbb{N}$. Similarly $v^i\leq
u^rz$ for some $z\in S$ and for all $r \in \mathbb{N}$, where $i$
is the smallest positive integer such that $v^i\in Reg_\leq(S)$.
So $u\rc^* v$. Hence $\rc\subseteq \rc^*$.

Therefore $c^{2mn}a^{2mn}\rc^* c^{2mn}b^{2mn}$. Hence $(ac)\rc^*
(c^{2mn}a^{2mn}) \rc^* (c^{2mn}b^{2mn}) \rc^* (bc)$, by (*). So
$ac\rc^* bc$. Also $ca \rc^*ac \rc^*bc \rc^* cb$, by (*). Hence
$\rc^*$ is a semilattice congruence on $S$.

$(4)\Rightarrow (1)$: This is obvious.
\end{proof}

\begin{Corollary}
Let $S$ be a $\pi$-inverse ordered semigroup. Then following
conditions are equivalent:
\begin{enumerate}
\item\vspace{-.4cm}
$\hc^*$ is a congruence on $S$;
\item\vspace{-.4cm}
 $\lc^*=\rc^* =\hc^*$;
 \item\vspace{-.4cm}

$S$ is a semilattice of $\pi$-$t$-simple ordered semigroups;
\item\vspace{-.4cm} $\hc^*$ is a semilattice congruence on $S$.
\end{enumerate}
\end{Corollary}

The following corollary which follows from  Theorem $\ref{4}$ and
Theorem $\ref{8}$ gives a characterization on right $\pi$-inverse
ordered semigroup to become a completely $\pi$-regular ordered
semigroup.
\begin{Corollary}
Let $S$ be a right $\pi$-inverse and left $\pi$-regular ordered
semigroup. Then following conditions are equivalent.
\begin{enumerate}
\item\vspace{-.4cm}
$\rc^*$ is a congruence on $S$;
\item\vspace{-.4cm}
$\lc^*\subseteq \rc^*$;
\item\vspace{-.4cm}
$S$ is a semilattice of right $\pi$-$t$-simple ordered semigroups;
\item\vspace{-.4cm}
$S$ is completely $\pi$-regular and left
weakly commutative.
\end{enumerate}
\end{Corollary}

\bibliographystyle{plain}

\begin{thebibliography}{10}
\baselineskip 5mm


\bibitem{bh1}
A. K. Bhuniya and K. Hansda, {On completely regular and Clifford
ordered semigroups}, \emph{Afr. Mat.}, \textbf{31}(2020),
1029-1045. doi:10.1007/s13370-020-00778-1



\bibitem{Cao 2000}
Y. Cao and X. Xinzhai, {Nil-extensions of simple po-semigroups},
\emph{Communication in Algebra}, \textbf{28(5)}(2000), 2477-2496.



\bibitem{HJ}
K. Hansda and A. Jamadar, {Characterization of inverse ordered
semigroups by their ordered idempotents and bi-ideals},
\emph{Quasigroups and Related Systems}, \textbf{28}(2020), 77-88.





\bibitem{JH}
A. Jamadar and K. Hansda, { On Right inverse ordered semigroups},
\emph{Discussiones Mathematicae- General Algebra and Applications},
\textbf{43(1)}(2024), 75-83.

\bibitem{AJ}
A. Jamadar, {$\pi$-inverse ordered semigroups}, \emph{Discussiones
Mathematicae- General Algebra and Applications},
\textbf{44(1)}(2024), 5-13.



\bibitem{Ke2006}
N. Kehayopulu, {Ideals and Green's relations in ordered
semigroups}, \emph{International Journal of Mathematics and
Mathematical Sciences }, \textbf{}(2006), 1-8, Article ID 61286.



\bibitem{ke92}
N. Kehayopulu,   On completely regular poe-semigroups, \emph{Math.
Japonica} \textbf{37}(1992), 123-130.

\bibitem{ke 2008}
N. Kehayopulu and M. Tsingelis,  {Semilattices of Archimedean
ordered semigroups}, \emph{Algera Colloquium} \textbf{15:3}(2008),
527-540.



\bibitem{SH}
S. Sadhya and K. Hansda, {Characterizations of $\pi$-$t$-simple
ordered semigroups by their ordered idempotents},
\emph{Quasigroups and Related Systems}, \textbf{27}(2019),
119-126.

\bibitem{SH1}
S. Sadhya and K. Hansda, { On Green's relations in  GV ordered
semigroups}, Communicated.


\end{thebibliography}

\end{document}